\documentclass[12pt,leqno,draft]{amsart}
\usepackage{amsmath,color}
\newcommand{\R}{{\mathbb R}}
\def\BbbP{{{\rm I} \!  {\rm P}}}
\def\BbbE{{{\rm I}  \! {\rm E}}}

\def \P{\BbbP}
\def \E{\BbbE}
\newcommand{\cal}{\mathcal}

\theoremstyle{plain}
\newtheorem{theorem}{Theorem}
\newtheorem{lemma}{Lemma}[section]
\newtheorem{corollary}[lemma]{Corollary}
\newtheorem{proposition}[lemma]{Proposition}

\theoremstyle{definition}
\newtheorem{remark}{Remark}

\begin{document}
\bibliographystyle{plain}
\title[Change point testing for a periodic Mean Reversion Process]{Change point 
testing for the drift parameters of a periodic Mean Reversion Process}
\author[H. Dehling]{Herold Dehling}
\author[B. Franke]{Brice Franke}
\author[T. Kott]{Thomas Kott}
\author[R. Kulperger]{Reg Kulperger}
\today
\address{Fakult\"at f\"ur Mathematik, Ruhr-Universit\"at Bochum,
44780 Bochum, Germany}
\email{herold.dehling@ruhr-uni-bochum.de}
\email{thomas.kott@ruhr-uni-bochum.de}
\address{D\'{e}partement de Math\'{e}matique, UFR Sciences et Techniques,
Universit\'{e} de Bretagne Occidentale, 29200 Brest, France}
\email{brice.franke@univ-brest.fr}
\address{Department of Statistical \&
Actuarial Sciences, University of Western Ontario,
London, N6A 5B7, Canada}
\email{kulperger@uwo.ca}


\begin{abstract}
In this paper we investigate  the problem of detecting a change in the drift parameters of a
generalized  Ornstein-Uhlenbeck process which is defined as the solution of
\[
 dX_t=(L(t)-\alpha X_t) dt + \sigma dB_t  
\]
and which is observed in continuous time. We derive an explicit representation 
of the  generalized likelihood ratio test statistic assuming that the mean reversion function 
$L(t)$ is a finite linear combination  of known basis functions. 
In the case of a  periodic mean reversion function, we determine the asymptotic distribution of the 
test statistic under the null hypothesis.
\end{abstract}
\maketitle

\section{Introduction}
The problem of  testing for a change in the parameters of a stochastic process has been an important 
issue in statistical inference for a long time. Initially investigated for i.i.d. data, change point analysis has more recently been extended to time series of dependent data. For a general review of change-point analysis, see e.g. the book by Cs\"org\H{o} and Horvath \cite{cso}.

In the present paper, we investigate the problem of detecting changes in the drift parameters of a diffusion process. Diffusion processes are a popular and widely studied class of models with applications in economics, finance, physics and engineering. Statistical inference for diffusion processes has been investigated by many authors, see e.g. the monographs by Liptser and Shiryaev \cite{lip} and by Kutoyants 
\cite{kuto}. However, change point analysis for diffusion processes has found little attention up to now.
We mention the papers by  Beibel \cite{Beibel}, Lee, Nishiyama and Yoshida \cite{Lee}, Mihalache \cite{Miha} and 
Negri and Nishiyama \cite{Negri}, where tests for the change points in the drift parameters of diffusions are discussed. 
However, all those papers deal with drift which stays constant in time before and after the eventual parameter change.

In our paper, we focus on change-point analysis for a special class of diffusion processes, namely for so-called generalized Ornstein-Uhlenbeck processes. These processes are defined as solutions to the stochastic differential equation
\begin{equation}\label{eq:gen-ou}
 dX_t=(L(t)-\alpha X_t)dt+\sigma dB_t , \quad t\geq 0,
\end{equation}
where $\alpha$ and $\sigma$ are positive constants and where the mean-reversion function $L(t)$ is 
non-random. $(B_t)_{t\geq 0}$ denotes standard Brownian motion and $X_0$ is a square-integrable real-valued random variable that is independent of $(B_t)_{t\geq 0}$. If $L(t)\equiv \mu$ is a constant, we obtain the classical Ornstein-Uhlenbeck process, introduced by Ornstein and Uhlenbeck \cite{ornstein}. 
Ornstein-Uhlenbeck processes are popular models for prices of commodities that exhibit a trend of reversion to a fixed mean level. Generalized Ornstein-Uhlenbeck processes can be used as models for the evolution of
prices with a trend or seasonal component $L(t)$.

Dehling, Franke and Kott \cite{De} have studied the problem of parameter estimation of a generalized Ornstein-Uhlenbeck process if the mean-reversion function $L(t)$ is a linear combination of known basis functions 
$\varphi_1(t),\ldots,\varphi_p(t)$, i.e. when
\begin{equation}\label{eq:l-fct}
 L(t)=\sum_{i=1}^p \mu_i \varphi_{i}(t)
\end{equation}
In this model, the unknown parameter vector is 
 $\theta=(\mu_1,\ldots,\mu_p,\alpha)^t$. We denote the corresponding parameter space by $\Theta$, and observe that
\[
 \Theta=\R^p\times (0,\infty).
\]
As is usual in the statistical inference for the drift of a time-continuously observed diffusion process, the diffusion parameter $\sigma$ is supposed to be known. This can be justified by the fact that the volatility $\sigma$ can be computed by the quadratic variation of the process. H\"opfner and Kutoyants (2010) study parameter estimation in more general diffusion models with unknown periodic signals 
$S(\theta,t), 0\leq t\leq T$, having a discontinuity at some unknown point.

We are interested in testing whether there is a change in the  values of 
the parameters $\mu_1,\ldots,\mu_p$ and $\alpha$, in the time interval $[0,T]$ during which the 
process is observed. In the first step, we will consider this problem assuming that the change-point 
$\tau\in (0,T)$ is known. For the asymptotic analysis, when $T\rightarrow \infty$, we write $\tau=s\, T$,
where $s\in (0,1)$ is known. The generalized Ornstein-Uhlenbeck process with change-point $\tau=s\,T$ is given by
\begin{equation}\label{eq:gen-drift}
 dX_t=(S(\theta,t,X_t)\mathbf{1}_{\{t\leq\tau\}} + 
S(\theta',t,X_t)\mathbf{1}_{\{t>\tau\}})dt+\sigma dB_t , \quad 0\leq t\leq T,
\end{equation}
where 
 \begin{equation}\label{eq:drift}
    S(\theta,t,X_t) = \sum_{i=1}^p \mu_i \varphi_{i}(t) -\alpha X_t,
\end{equation}
and where $\mathbf{1}_{A}$ denotes the indicator function of the set $A$. 
The test problem of interest can be formulated as
\begin{equation}\label{eq:test-hypo}
  H_0: \textrm{ }\theta=\theta' \textrm{ (no change) } \quad\textrm{vs.}\quad 
H_A:\textrm{ } \theta\neq\theta'  \textrm{ (change at time point $\tau$) }.
\end{equation}

We want to study the  generalized likelihood ratio test for this test problem. 
We denote by $P_X$ the measure induced by the observable realizations 
$X^T=\{X_t,0\leq t \leq T\}$  on the measurable space $\left(C[0,T],\mathcal{B}[0,T]\right)$, 
$C[0,T]$ being  the space of continuous, real-valued functions on $[0,T]$ and 
$\mathcal{B}[0,T]$ the associated Borel $\sigma$-field. Moreover, let $P_B$ be the 
measure generated by the Brownian motion on $\left(C[0,T],\mathcal{B}[0,T]\right)$. 
Then the likelihood function $\mathcal{L}$ of observations $X^T$ of the process with 
stochastic differential (\ref{eq:gen-drift}) is defined as the Radon-Nikodym derivative, i.e.
\begin{equation*}
 \mathcal{L}(\theta,\theta^\prime,X^T): =\frac{dP_X}{dP_B}(X^T).
\end{equation*}
The generalized  likelihood ratio  $\mathcal{R}(X^T)$ is given by
\begin{equation}\label{eq:lik-test}
   \mathcal{R}(X^T) = \frac{\sup_{\theta\in \Theta}\,
\mathcal{L}(\theta,\theta,X^T)}{\sup_{\theta,\theta^\prime\in \Theta}\mathcal{L}(\theta,\theta^\prime,X^T)}.
\end{equation}
Note that this likelihood ratio depends on the suspected change point  $\tau=s T$, where $s\in (0,1)$. 
Eventually, we will study the log-transformed likelihood ratio 
\[
 \Lambda_T(s):=-2\log (\mathcal{R}(X^T)).
\]
We will give an explicit expression of the process $(\Lambda_T(s))_{0\leq s\leq 1}$ and study the asymptotic
distribution of this process as $T\rightarrow \infty$.

With the above choice of the log-transformed likelihood ratio test statistic $ \Lambda_T $ we put our focus on the optimization of the power of the change point test. 
Another interesting problem, which appears in sequential testing for change points, is the optimization of the time needed for detection of the change point. This is usually realized through some CUSUM test statistic as presented in Siegmund and Venkatraman 
\cite{SigVen} or in the book of Siegmund \cite{Siegmund}. Further, one should note that sequential testing for changes in drift was studied by Beibel \cite{Beibel}, Lee, Nishiyama and Yoshida \cite{Lee} and in Mihalache \cite{Miha}. 
In this paper our focus does not lie in sequential change point detection and thus we do not discuss CUSUM test in what follows.
The approach that we intend to follow in this paper has been developped by Davis, Huang and Yao \cite{davis} for detecting
change points in autoregressive models. Since our model can be considered as some sort of time-continous version of an 
autoregressive process with a time-periodic mean, it is natural to use the theory presented there as a guideline for our research. 

The outline of the paper is as follows. In Section~2, we will first derive an explicit representation of the 
log-transformed generalized likelihood ratio test statistic $(\Lambda_T(s))_{0\leq s\leq 1}$. We then 
formulate two theorems concerning the asymptotic distributions of sup-norm functionals of $(\Lambda_T(s))_{0\leq s\leq 1}$. It turns out that $\sup_{0\leq s\leq 1} \Lambda_T(s)$ does not have a
non-degenerate limit distribution, as $T\rightarrow \infty$.  In Theorem~1, we will prove convergence of $\sup_{s_1\leq s\leq s_2} \Lambda_T(s)$,
when $0<s_1<s_2<1$ are fixed constants. In Theorem~2, we will show that there exist centering and norming sequences $a_T$ and $b_T$ such that 
$(\sup_{0\leq s\leq 1} \Lambda_T(s) - b_T)/a_T$ converges towards an extreme value distribution. The proofs of these theorems are given in Section~3 and Section~4, respectively.


\section{Generalized likelihood ratio test}\label{sec:gen-lik}

In this section, we will derive an explicit representation of $\Lambda_T(s)$. 
In order to do so, we  need to calculate the maxima in  the numerator and denominator  in (\ref{eq:lik-test}). 
Note that this is achieved by the corresponding maximum likelihood estimators. 
A corollary to Girsanov's theorem, see  Theorem 7.6 in  Lipster and Shiryayev \cite{lip}, gives an explicit expression of the likelihood 
function of a diffusion process provided that
\begin{equation}\label{cond1}
 \P\left(\int_0^T S(\theta,t,X_t)^2dt<\infty\right)=1 
\end{equation}
for all $0\leq T <\infty$ and all $\theta$.

\begin{lemma}\label{pr:ml-es}
Let $\mathcal{L}(\theta,X^T)$ denote the likelihood function of the
observations $X^T=\{X_t,0\leq t \leq T\}$ of the generalized Ornstein-Uhlenbeck process $(X_t)_{t\geq 0}$,
defined in  (\ref{eq:gen-ou}), with mean reversion function (\ref{eq:l-fct}).
If the drift term  (\ref{eq:drift}) satisfies condition (\ref{cond1})   then
\begin{equation*}
    \arg \max_{\theta} \mathcal{L}(\theta,X^T) = \hat{\theta}_{ML}=Q_T^{-1} \tilde{R}_T.
\end{equation*}
Here $Q_T \in \R^{(p+1)\times (p+1)}$ and
$ \tilde{R}_T  \in \R^{p+1}$ are defined as
\begin{eqnarray*}
 Q_T &=& \left(
  \begin{array}{cc}
  G_T & -a_T \\
  -a^t_T & b_T
  \end{array}
  \right),
\label{eq:q-matrix}\\
 \tilde{R}_T&=&\left(
  \int_0^T \varphi_1(t) d X_t, \ldots, \int_0^T \varphi_p(t) d X_t,
 - \int_0^T X_t\, d X_t
  \right)^t,
\label{eq:p-vektor}
\end{eqnarray*}
where $
 G_T =(\int_0^T \varphi_j(t)
 \varphi_k(t)  dt )_{1\leq j,k \leq p} \in
\R^{p\times p}$,
$a_T=(\int_0^T \varphi_1( t)X_{t}\, dt,
\ldots, \int_0^T \varphi_p(t)X_{t}\, dt )^t$ and
$b_T=\int_0^T X_{t}^2 dt$.

\end{lemma}

\begin{remark}
The expression for the likelihood estimator $ \hat{\theta}_{ML} $ relies on the invertibility of the matrix $ Q_T $.
It was proved in Dehling, Franke and Kott (\cite{De}, p.~180) that as $ T $ becomes large the matrix $ Q_T $ becomes invertible almost surely.
\end{remark}

\begin{remark}
Note that the integrals in $\tilde{R}_T$ can be rewritten as
\begin{equation*}
  \int_0^T\varphi_i(t) d X_t=\int_0^T \varphi_i(t)(L(t)-\alpha X_t)\,dt 
+ \sigma\int_0^T\varphi_i(t)dB_t
\end{equation*}
where the latter is a well-defined It\^{o} integral.
\end{remark}

\begin{proof}
The likelihood function $\mathcal{L}$ of a general diffusion process
\begin{equation*}
  dX_t = S(\theta,t,X_t)dt + \sigma dB_t, \quad 0\leq t \leq T,
\end{equation*}
is  given by
\begin{equation}\label{eq:lik-fct}
\mathcal{L}(\theta,X^T) =\frac{dP_X}{dP_B}(X^T) 
= \exp\left( \frac{1}{\sigma^2}\int_0^T S(\theta,t,X_t)dX_t -
\frac{1}{2\sigma^2}\int_0^T S(\theta,t,X_t)^2dt\right)
\end{equation}
if condition (\ref{cond1}) is fulfilled; see Theorem~7.6 in Lipster and Shiryayev \cite{lip}.
The maximum likelihood estimator is  defined as the maximum of the functional 
$\theta \mapsto \mathcal{L}(\theta,X^T)$
and  the partial derivatives of the logarithm of this functional are
\begin{equation}\label{eq:der-ze}
   \frac{\partial}{\partial\theta_i}\ln({\mathcal{L}}(\theta,X^T))
= \frac{1}{\sigma^2}\int_0^T \frac{\partial}{\partial\theta_i} S(\theta,t,X^T)dX_t
-  \frac{1}{\sigma^2}\int_0^T S(\theta,t,X_t)\frac{\partial}{\partial\theta_i}
S(\theta,t,X_t)dt.
\end{equation}
The derivatives of the drift function specified in (\ref{eq:drift}) can be computed to be
\begin{equation*}
    \frac{\partial}{\partial\theta_i} S(\theta,t,X_t) = \left\{
       \begin{array}{ll}
         \varphi_{i}(t), & \hbox{$i=1,\ldots,p$;} \\
         -X_t , & \hbox{$i=p+1$.}
       \end{array}
     \right.
\end{equation*}
Setting the partial derivatives of the log-likelihood function in (\ref{eq:der-ze})
equal zero gives  a system of linear equations which yields the assertion. 
\end{proof}

Due to the linearity of the drift term,
the  log-likelihood function of the process (\ref{eq:gen-drift}) is given by
\begin{eqnarray*}
  \ln \left(\mathcal{L}(\theta,\theta',X^T)\right) 
&=&  \frac{1}{\sigma^2}\left(\int_0^{\tau} S(\theta,t,X_t)dX_t 
         + \int_{\tau}^T S(\theta',t,X_t)dX_t\right) \\
   & &  -
\frac{1}{2\sigma^2} \left(\int_0^{\tau} S(\theta,t,X_t)^2dt 
         +\int_{\tau}^T S(\theta',t,X_t)^2dt \right).
\end{eqnarray*}
Hence, defining $X^{\tau,T}=\{X_t,\tau\leq t\leq T\}$, we can write the generalized likelihood ratio  (\ref{eq:lik-test})  as
\begin{equation}\label{eq:lfct}
    \mathcal{R}(X^{T}) = 
 \frac{\sup_{\theta}\mathcal{L}(\theta,X^{T})}{\sup_{\theta^{\ast}} 
   \mathcal{L}(\theta^{\ast},X^{\tau})
    \sup_{\theta'}\mathcal{L}(\theta',X^{\tau,T})}
\end{equation}
where $\mathcal{L}(\theta,X^{T})$ is given in (\ref{eq:lik-fct}) with drift function 
  specified in (\ref{eq:drift}).
The  terms $\mathcal{L}(\theta^{\ast},X^{\tau})$ and $\mathcal{L}(\theta',X^{\tau,T})$ 
are defined analogously as integrals  with integration regions 0 to $\tau$ and $\tau$ to $T$, respectively.  
It follows from Lemma \ref{pr:ml-es} that
\begin{equation}\label{eq:lh-est}
    \mathcal{R}(X^{T})=\frac{\mathcal{L}(\hat{\theta}_{ML}^{T},X^{T})}
{\mathcal{L}(\hat{\theta}_{ML}^{\tau},X^{\tau}) \mathcal{L}(\hat{\theta}_{ML}^{\tau,T},X^{\tau,T})}
\end{equation}
where the maximum likelihood estimates $\hat{\theta}_{ML}^{T}$, $\hat{\theta}_{ML}^{\tau}$ 
and $\hat{\theta}_{ML}^{\tau,T}$ are computed from the total, the pre- and post-change sample, 
respectively.
This representation of the likelihood ratio is used to prove the following proposition.

\begin{proposition}\label{pr:glr}
The log-transformed generalized likelihood ratio test statistic $\Lambda_T(s)=-2\ln(\mathcal{R}(X^T))$
of the test problem (\ref{eq:test-hypo}) can be represented under the null hypothesis as
\begin{equation*}\label{ratio}
    \Lambda_T(s)= - R_{T}^t Q_{T}^{-1}R_{T} + R_{\tau}^t Q_{\tau}^{-1}R_{\tau} 
+ (R_{T} -R_{\tau})^t (Q_{T}-Q_{\tau})^{-1}(R_{T} -R_{\tau})
\end{equation*}
where $Q_{T}$ is given in Lemma  \ref{pr:ml-es}  and
\begin{equation*}
  R_{T} = \left(  \int_0^{T} \varphi_1(t)dB_{t},
    \dots , \int_0^{T} \varphi_p( t)dB_{ t}   , -\int_0^{T} X_{ t}dB_{ t} \right)^t.
\end{equation*}
\end{proposition}

\begin{proof}
Our aim is to compute an explicit expression of the ratio given in (\ref{eq:lh-est}).
Note that the likelihood function in the numerator of (\ref{eq:lfct}) can be represented as
\begin{equation*}
    L(\theta,X^T)=\exp\left( \frac{1}{\sigma^2}\theta^t\tilde{R}_T 
- \frac{1}{2\sigma^2}\theta^tQ_T\theta\right)
\end{equation*}
where $\tilde{R}_T$ and $Q_T$ are given in Lemma \ref{pr:ml-es}.
Denoting by $\theta_0$  the true  value of $\theta$, the representations
\begin{equation*}
    \hat{\theta}_{ML}^T =Q_T^{-1}\tilde{R}_T \quad\textrm{and}\quad \tilde{R}_T = Q_T \theta_0 + \sigma R_T,
\end{equation*}
where the latter can be obtained by plugging in the initial SDE
(\ref{eq:gen-ou}) and (\ref{eq:l-fct}), lead to
\begin{eqnarray*}
  L(\hat{\theta}_{ML}^{T},X^{T}) &=& \exp\left(\frac{1}{2\sigma^2}\tilde{R}_T^t Q_T^{-1}\tilde{R}_T   \right) \\
   &=& \exp\left( \frac{1}{2\sigma^2} \theta_0^tQ_T\theta_0 + \frac{1}{\sigma} R_T^t\theta_0 
+ \frac{1}{2}R_T^t Q_T^{-1} R_T \right).
\end{eqnarray*}
The same procedure yields an analog expression for $L(\hat{\theta}_{ML}^{\tau},X^{\tau})$. 
The additivity of  the integrals provides
\begin{equation*}
    \hat{\theta}_{ML}^{\tau,T}=(Q_{T}-Q_{\tau})^{-1}(\tilde{R}_{T}-\tilde{R}_{\tau})
\end{equation*}
and
\begin{equation*}
    \tilde{R}_{T}-\tilde{R}_{\tau} = (Q_{T}-Q_{\tau}) \theta_0 + \sigma (R_{T}-R_{\tau})
\end{equation*}
such that
\begin{eqnarray*}
  L(\hat{\theta}_{ML}^{\tau,T},X^{\tau,T}) &=& \exp\bigg( \frac{1}{2\sigma^2} 
\theta_0^t(Q_T-Q_{\tau})\theta_0 +  \frac{1}{\sigma} (R_T-R_{\tau})^t\theta_0\\
    & & +  \frac{1}{2}(R_T-R_{\tau})^t (Q_T-Q_{\tau})^{-1} (R_T-R_{\tau})\bigg).
\end{eqnarray*}
Under the null hypothesis, cancelation of several terms in (\ref{eq:lh-est}) 
proves the assertion.
\end{proof}


For the  rest of our investigations, we study  periodic functions
\begin{equation}\label{eq:phi-per}
 \varphi_j(t+\nu)=\varphi_j(t)
\end{equation}
where $\nu$ is the period observed in the data. Under the null hypothesis of no 
change, this implies periodicity of the mean reversion function, i.e. $L(t+\nu)=L(t)$.
We assume that we observe the process over some integer multiple of  periods, i.e.  $T=n \nu$, 
$n\in\mathbb{N}$. By Gram-Schmidt orthogonalization we may assume without loss of
generality that the basis functions $\varphi_1(t),\ldots,\varphi_p(t)$ form an orthonormal 
system in $L^2([0,\nu], \frac{1}{\nu} d\lambda )$, i.e. that
\begin{equation}
\label{eq:phi-ons}
\int_0^\nu \varphi_j(t)\varphi_k(t)dt=
\left\{
  \begin{array}{ll}
    \nu , &\quad j=k  \\
    0, &\quad  j\neq k.
  \end{array}
\right.
\end{equation}
Under these assumptions, the matrix $Q_T$ appearing in the test statistic $\Lambda_T(s)$, 
see Proposition \ref{pr:glr}, simplifies to
\begin{equation*}
 Q_T = \left(
  \begin{array}{cc}
  T\, I_{p\times p} & a_T \\
  a_T^t & b_T
  \end{array}
  \right).
\end{equation*}


\begin{theorem}\label{theo-1}
Let  $X^T=\{X_t,0\leq t \leq T\}$ be observations of the  mean
reversion process (\ref{eq:gen-ou}) with mean reversion function of the form
 (\ref{eq:l-fct}),
satisfying (\ref{eq:phi-per}) and (\ref{eq:phi-ons}). Denote by 
$\Lambda_T(s)=-2\ln(\mathcal{R}(X^T))$
the log-transformed generalized likelihood ratio test statistic for the test problem (\ref{eq:test-hypo}). 
Then, for any fixed $0<s_1<s_2<1$, under the null hypothesis,
\begin{equation*}
    \sup_{s\in[s_1,s_2]} \Lambda_T(s)  \stackrel{\mathcal{D}}{\longrightarrow} 
\sup_{s\in[s_1,s_2]} \frac{\| W(s)-sW(1)\|^2}{s(1-s)}
\end{equation*}
as $T\rightarrow \infty$. Here $\|\cdot \|$ denotes the Euclidean 
norm and $W$ is a $(p+1)$-dimensional standard Brownian motion.
\end{theorem}

The previous theorem has its counterpart in the theory of change point detection in an autoregressive 
sequence presented in Davis, Huang and Yao (see \cite{davis} Remark 2.2).
It is also noted there, that results from the type of Theorem~\ref{theo-1} are not satisfactory in application:
First, it is not clear how to choose the interval $[s_1,s_2]$ potentially containing a 
change point if  no information about the location of the change is available.
Second, the distribution of the limit which is the squared length of a multi-dimensional 
Brownian bridge  is not explicitly given such that further  analysis or a simulation 
study are necessary in order to specify the limit distribution explicitly. In order 
to avoid this inconvenience it is possible to consider the \emph{exact} test statistic 
$\sup_{0<s\leq 1} \Lambda_T(s)$. 
This was done by Davis, Huang and Yao in the situation of change point detection in an autoregressive model 
(see \cite{davis} Theorem 2.2). 
For this situation they proved that after some renormalization the log-transformed generalized likelihood ratio test statistic
converges in law to the Gumbel distribution. We have an analogous result for the change point detection in a periodic mean 
reversion process.

\begin{theorem}\label{theo-gumbel}
Under the same assumptions as in Theorem~\ref{theo-1}
it holds under the null hypothesis that
\begin{equation*}
  \big(\sup_{0<s\leq 1} \Lambda_T(s) - b_T\big)/a_T \stackrel{\mathcal{D}}{\longrightarrow} G,
\end{equation*}
as $T\rightarrow \infty$, where $G$ denotes a real-valued random variable satisfying
\begin{equation*}
    \P(G\leq x) = \exp(-2e^{-x/2}).
\end{equation*}
Here $b_T=\left(2\ln\ln \frac{T}{\nu} + \frac{p+1}{2}\ln\ln\ln \frac{T}{\nu} 
- \ln \Gamma(\frac{p+1}{2})\right)^2/(2\ln\ln \frac{T}{\nu})$,  $a_T=\sqrt{b_T/(2\ln\ln \frac{T}{\nu})}$ 
where $\Gamma$ is the gamma function.

\end{theorem}


\section{Proof of Theorem 1}
Before we can complete the proof of Theorem \ref{theo-1}, we have to establish some 
auxiliary results. First, we will study the asymptotic behavior of $\Lambda_T(s)$
in the case of a periodic mean reversion function, see (\ref{eq:phi-per}) 
and (\ref{eq:phi-ons}). Note that  by
Proposition \ref{pr:glr} we have the representation
 \begin{equation*}
    \Lambda_T(s)= - R_{T}^t Q_{T}^{-1}R_{T} + R_{sT}^t Q_{sT}^{-1}R_{sT} 
+ (R_{T} -R_{sT})^t (Q_{T}-Q_{sT})^{-1}(R_{T} -R_{sT}).
\end{equation*}
The first term,
\begin{equation*}
    R_{T}^t Q_{T}^{-1}R_{T}= \frac{1}{\sqrt{T}}R_{T}^t \Big(\frac{1}{T}Q_{T}\Big)^{-1}
\frac{1}{\sqrt{T}}R_{T},
\end{equation*}
has already been studied by Dehling, Franke and Kott \cite{De}. The following proposition summarizes the 
results of  Proposition 4.5, 5.1 and 5.2 in Dehling et al. \cite{De}.

\begin{proposition}\label{prop:old-pap}
We have that
\begin{equation*}
   \frac{1}{\sqrt{T}}R_{T} \xrightarrow{\mathcal{D}} N(0,\Sigma)
\end{equation*}
and
\begin{equation*}
  \frac{1}{T}Q_{T}\rightarrow\Sigma\textrm{, almost surely,}
\end{equation*}
as $T\rightarrow\infty$. The matrix $\Sigma$ is given by
 \begin{equation}\label{def:sigma}
 \Sigma = \left(
  \begin{array}{cc}
  \nu\, I_{p\times p} & \Lambda \\
  \Lambda^t & \omega
  \end{array}
  \right)
\end{equation}
where $ \Lambda:=(\Lambda_1,...,\Lambda_p)^t $ with $\Lambda_i=\int_0^{\nu}\varphi_i(t)\tilde{h}(t)dt$, $i=1,\ldots,p$, 
$\omega=\int_0^{\nu}(\tilde{h}(t))^2dt+\frac{\nu\sigma^2}{2\alpha}$ and
where $\tilde{h}:[0,\infty)\rightarrow \R$ is defined by
\begin{equation}\label{def:h-tilde}
  \tilde{h}(t) = e^{-\alpha t}\sum_{j=1}^p\mu_j\int_{-\infty}^t
  e^{\alpha s}\varphi_j(s)ds.
\end{equation}
Here, $N(0,\Sigma)$ denotes a normally distributed random vector with zero-mean and 
covariance matrix $\Sigma$.
\end{proposition}

Now we want to investigate the second term of $\Lambda_T(s)$ which we rewrite as
\begin{equation*}
    R_{sT}^t Q_{sT}^{-1}R_{sT}= \frac{1}{\sqrt{T}}R_{sT}^t 
\Big(\frac{1}{T}Q_{sT}\Big)^{-1}\frac{1}{\sqrt{T}}R_{sT}.
\end{equation*}
We will show that the process $\big(\frac{1}{\sqrt{T}}R_{sT}\big)_{s\in [s_1,s_2]}$ 
converges in distribution to a Gaussian process on $[s_1,s_2]$, and that $\frac{1}{T}Q_{sT}$  
converges in probability uniformly on $[s_1,s_2]$.

We need the following functional version of the asymptotic normality proved in \cite{De}

\begin{proposition}\label{prop:rtau}
As $ T\rightarrow\infty$, the sequence of processes $(R^{(T)}_s)_{0\leq s\leq 1}$, where 
\begin{equation*}
    R^{(T)}_s:=\frac{1}{\sqrt{T}}R_{sT},
\end{equation*}
converges in distribution to a $(p+1)$-dimensional Wiener-process $ R^\ast $  with 
$ R_s^\ast \sim N(0,s\Sigma)  $  and  where $\Sigma$ is defined in
 (\ref{def:sigma}). Thus the covariance function of the process $ R^\ast =(R^{\ast,1},...,R^{\ast,p+1})^t $ is of the form 
$$ {\rm Cov}\left(R^{\ast,i}_s,R^{\ast,j}_t\right) = (s\wedge t)\Sigma_{ij}; \ \ \  
\mbox{for} \ i,j=1,\ldots,p+1.$$
\end{proposition}

\begin{proof}
Remember that $ T=n\nu $ and note that the vector-valued processes
$$   R^{(T)}_t=\Bigg(\frac{1}{\sqrt{n}}\int_0^{n\nu t}\varphi_1(s)dB_s,..., 
\frac{1}{\sqrt{n}}\int_0^{n\nu t}\varphi_p(s)dB_s,
     \frac{1}{\sqrt{n}}\int_0^{n\nu t}X_sdB_s \Bigg) $$
are martingales with respect to the filtrations $ {\cal F}^{(T)}_t:=\sigma(B_s;s\leq n\nu t) $.
The associated covariance processes are given by
\[     
\langle  R^{(T)}, R^{(T)}\rangle_t = \left(\begin{array}{cccc}
       \frac{1}{n}\int_0^{n\nu t}\varphi_1\varphi_1ds  & ... &  
\frac{1}{n}\int_0^{n\nu t}\varphi_1\varphi_pds &  
       \frac{1}{n}\int_0^{n\nu t}\varphi_1 X_sds \\
\vdots &  & \vdots & \vdots \\ 
 \frac{1}{n}\int_0^{n\nu t}\varphi_p\varphi_1ds  & ... &  
\frac{1}{n}\int_0^{n\nu t}\varphi_p\varphi_pds & 
     \frac{1}{n}\int_0^{n\nu t}\varphi_p X_sds \\[2mm]
 \frac{1}{n}\int_0^{n\nu t}\varphi_1X_sds  & ... &  \frac{1}{n}\int_0^{n\nu t}\varphi_pX_sds &  
\frac{1}{n}\int_0^{n\nu t} X^2_sds 
\end{array}\right). 
\]
As was shown in \cite{De} (see p.184), these matrices converge for $ n\rightarrow \infty $ 
almost surely towards  the matrix
\[   
\left(\begin{array}{cccc}
      t \int_0^{\nu}\varphi_1\varphi_1ds  & ... &  t \int_0^{\nu}\varphi_1\varphi_pds &  
     t  \int_0^{\nu}\varphi_1 \tilde{h}ds \\
\vdots &  & \vdots & \vdots \\ 
    t \int_0^{\nu}\varphi_p\varphi_1ds  & ... & t \int_0^{\nu}\varphi_p\varphi_pds & 
   t  \int_0^{\nu}\varphi_p \tilde{h}ds \\[2mm]
  t  \int_0^{\nu}\varphi_1\tilde{h}ds  & ... & t \int_0^{\nu}\varphi_p\tilde{h}ds &  
t\left( \int_0^{\nu} \tilde{h}^2ds+\frac{\nu\sigma^2}{2\alpha}\right) \\
\end{array}\right)   =t\Sigma. 
\]
The functional central limit theorem for continuous martingales (p.339 in \cite{EthKur}) now 
implies that the sequence of continuous $ {\cal F}^{(n)}_t  $-martingales
$ R^{(n)} $ converges in distribution toward the unique continuous Gaussian martingale with 
covariance function $ t\Sigma $.
\end{proof}

\begin{proposition}\label{prop:qtau}
Let $Q_t$ be defined as in Lemma \ref{pr:ml-es}. Then,  as $T\rightarrow\infty$,
\begin{equation*}
    \frac{1}{T}Q_{sT} \longrightarrow s\Sigma
\end{equation*}
almost surely uniformly on $[0,1]$, where $\Sigma$ is given in (\ref{def:sigma}).
\end{proposition}

\begin{proof} 
By Proposition \ref{prop:old-pap}, we know that, almost surely, 
$ \frac{1}{T}Q_T\rightarrow\Sigma$ as $ T \rightarrow\infty $ . Thus, given $\epsilon >0$, there 
exists a $T_0$ such that for all $T\geq T_0$
\[
  \| \frac{1}{T} Q_T -\Sigma \| \leq \epsilon.
\]
Let $B:=\sup_{0\leq s\leq T_0} \|Q_s\|$. Then we get for any $T\geq T_0$ and $T_0/T\leq s\leq 1$
\[
  \| \frac{1}{T} Q_{Ts}-s\Sigma \| = s\|\frac{1}{Ts}Q_{Ts} -\Sigma  \| \leq \epsilon.
\]
For $s\leq T_0/T$ we obtain
\[
  \| \frac{1}{T} Q_{Ts}-s\Sigma \| \leq \frac{1}{T} B +\frac{T_0}{T} \|\Sigma\| \leq \epsilon,
\]
for $T$ large enough. The last two inequalities together show that for $T$ large enough, we have
$\| \frac{1}{T} Q_{Ts}-s\Sigma \| \leq \epsilon$, and this proves the statement of the proposition.
\end{proof}

We can finally finish  the proof of  Theorem~\ref{theo-1}, which essentially follows the method described in 
Davis, Huang and Yao (see \cite{davis}, p.8).

\begin{proof}[Proof of Theorem 1]
By Slutsky's theorem and Propositions~\ref{prop:old-pap}, \ref{prop:rtau} 
and \ref{prop:qtau} we obtain
\begin{eqnarray*}
  \Lambda_T(s) &=& - R_{T}^t Q_{T}^{-1}R_{T} + R_{sT}^t Q_{sT}^{-1}R_{sT} 
+ (R_{T} -R_{sT})^t (Q_{T}-Q_{sT})^{-1}(R_{T} -R_{sT})\\
   &\xrightarrow{\mathcal{D}}& - \|W(1) \|^2 + \frac{\|W(s) \|^2}{s} 
+ \frac{\|W(1) -W(s) \|^2}{1-s} \\
    &=& \frac{\| W(s)-s W(1)\|^2}{s(1-s)}
\end{eqnarray*}
in $C[s_1,s_2]$. 
Here we have used the fact that the process $ W(t):=\Sigma^{-1/2}R^\ast_t $ is a Brownian motion with covariance matrix $ I_{p+1} $, where $ I_{p+1} $ is the $(p+1)\times (p+1)$ identity matrix, and that
$ (R_t^\ast)^t\Sigma^{-1}R^\ast_t=\|W(t)\|^2 $. 
The assertion about the supremum of $\Lambda_T(s)$ is justified 
by the continuous mapping theorem.
\end{proof}

\section{Proof of Theorem 2}

The proof of Theorem~2 is motivated by the proof of an analogous result for discrete time AR processes,
given by Davis et al. \cite{davis}. 
We need the following result which is proved in Davis et al. \cite{davis} and which 
relies on  Lemma 2.2 in Horv\'{a}th~\cite{hor}).

\begin{proposition}[Corollary A.2 in Davis et al. \cite{davis}]\label{le:cor-davis}
Let $ Y_1,Y_2,\ldots $ be an i.i.d. sequence of $ (p+1)$-dimensional random vectors with 
$ \E[Y_1]=0 $ and $ \E[Y_1Y_1^t]=I_{p+1} $. Define $ S_k=\sum_{i=1}^k Y_i $. If
$ \max_{1\leq i\leq p+1} \E|Y_{i,1}|^{2+r}<\infty $ for some $ r>0 $, then
\begin{equation*}
  \big(\,\max_{1\leq k \leq n} \|S_k\|^2 - b_n\big)/a_n 
  \stackrel{\mathcal{D}}{\longrightarrow} G^*,
\end{equation*}
as $n\rightarrow \infty$, where $G^*$ denotes a real-valued random 
variable satisfying
\begin{equation*}
    P(G^*\leq x) = \exp(-e^{-x/2}).
\end{equation*}
Thereby, it is $b_n=\left(2\ln\ln n + \frac{p+1}{2}\ln\ln\ln n 
- \ln \Gamma(\frac{p+1}{2})\right)^2/(2\ln\ln n)$, $a_n=\sqrt{b_n/(2\ln\ln n)}$.
\end{proposition}

Recall that
\begin{equation*}
    \Lambda_T(s)= - R_{T}^t Q_{T}^{-1}R_{T} + R_{sT}^t Q_{sT}^{-1}R_{sT} 
+ (R_{T} -R_{sT})^t (Q_{T}-Q_{sT})^{-1}(R_{T} -R_{sT})
\end{equation*}
and $T=n\nu$, $\nu$ fixed. Let us assume for a moment that $ \nu = 1 $.
We write  $\Lambda_n(s)$, $R_{n}$ and $Q_{n}$ 
for $\Lambda_T(s)$, $R_{T}$ and $Q_{T}$, respectively, since  the asymptotic 
framework is $n\rightarrow \infty$.

In order to simplify the investigation of the process $ (X_t)_{t\geq0} $ we introduce some auxiliary process
$$ \tilde{X}_t:= \tilde{h}(t)+\tilde{Z}_t $$
with
$$ \tilde{h}(t):=e^{-\alpha t}\int_{-\infty}^te^{\alpha s}\sum_{j=1}^p\mu_j\varphi_j(s)ds    $$
and
$$ \tilde{Z}_t=\sigma e^{-\alpha t}\int_{-\infty}^t e^{\alpha s}d\tilde{B}_s $$
where  
$$ \tilde{B}_s:=\check{B}_s{\bf 1}_{\mathbb{R}_+}(s)+\hat{B}_{-s}{\bf 1}_{\mathbb{R}_-}(s) $$
is bilateral Brownian motion defined through two independent standard Brownian motions
$ (\check{B}_t)_{t\geq0} $  and $ (\hat{B}_t)_{t\geq0} $. 
First we note that this process is also a solution to the stochastic differential equation (\ref{eq:gen-ou}), which does not depend on 
$ X_0 $. Furthermore, it was proved in \cite{De} (Lemma 4.4) that one has 
$|\tilde{X}_t-X_t |\rightarrow 0$, almost surely, as $ t\rightarrow\infty $. This fact will give us the possibility to reduce the 
asymptotic analysis on $ (X_t)_{t\geq0} $ to the one of $ (\tilde{X}_t)_{t\geq0} $. The essential feature of the new process 
$ (\tilde{X}_t)_{t\geq0} $ is that we can chop it into peaces which form a stationary sequence of $ C[0,1] $-valued random 
variables $ (\tilde{X}_{k-1+s})_{s\in[0,1]}; k\in\mathbb{N} $ (see Dehling, Franke, Kott \cite{De}, Lemma 4.3).
This gives us the possibility to use methods from ergodic theory in the investigation of $ (X_t)_{t\geq0} $. In the following we 
assume without loss of generality that $ (X_t)_{t\geq0} $ has the above representation in terms of $ \tilde{h} $ and 
$ (\tilde{Z}_t)_{t\geq0} $.

The following proposition is  essential  for the proof of Theorem~\ref{theo-gumbel}. 
First, define for two $\sigma$-algebras $\mathcal{A}$ and $\mathcal{B}$  the quantities
\begin{equation*}
    \alpha(\mathcal{A},\mathcal{B}) = \sup_{A\in\mathcal{A}, 
B\in\mathcal{B}} |\P(A\cap B) -\P(A)\P(B)|
\end{equation*}
and
\begin{equation*}
  \rho(\mathcal{A},\mathcal{B}) = \sup_{F\in L^2(\mathcal{A},\P),G\in L^2(\mathcal{B},\P)}
    {\rm Corr}(F,G).
\end{equation*}
It is known that 
$$  \alpha(\mathcal{A},\mathcal{B})\leq   \frac{1}{4}\rho(\mathcal{A},\mathcal{B})  . $$
For a stationnary sequence of random variables $(\zeta_k)_{k\in\mathbb{N}}$ define 
the mixing coefficient $ \alpha_\zeta $ by
\begin{equation*}
    \alpha_\zeta(n) = \sup_{k\in\mathbb{N}}  \alpha\big(\sigma(\zeta_i;i\leq k),
\sigma(\zeta_i;i\geq k+n)\big).
\end{equation*} 
The sequence $(\zeta_k)_{k\in\mathbb{N}}$ is called strongly mixing 
if $\alpha_\zeta(n)\rightarrow 0$ as $n\rightarrow\infty$.

\begin{proposition} \label{Prop-R_k-mixing}
The sequence of random vectors $(\Delta R_k)_{k\in\mathbb{N}}$ defined by
\begin{equation*}
  \Delta R_{k} :=(R_k-R_{k-1})= \left(  \int_{k-1}^{k} \varphi_1(t)dB_{t},
    \dots , \int_{k-1}^{k} \varphi_p( t)dB_{ t}   , -\int_{k-1}^{k} X_{t}dB_{t} \right)^t
\end{equation*}
is strongly mixing with mixing coefficient $\alpha$ of order
\begin{equation*}
    \alpha_{\Delta R}(n) = \mathcal{O}(e^{-\alpha(n-1)}).
\end{equation*}
\end{proposition}

\begin{proof}
Define the $ C[0,1] $-valued stochastic process $(\xi^{(k)})_{k\in\mathbb{N}}$ by 
\begin{equation*}
    \xi^{(k)} :=\left(\begin{array}{c} X^{(k)} \\
                               B^{(k)} \end{array}\right) := \left(
              \begin{array}{c}
                (X_{t+k})_{t\in[0,1]} \\
                (B_{t+k})_{t\in[0,1]} \\
              \end{array}
            \right). 
\end{equation*}
The process $(\xi^{(k)})_{k\in\mathbb{N}}$ is both a Markov and a Gaussian process. Hence, 
by making use of the Markov property and by applying the correlation inequality for 
Gaussian processes from Lemma \ref{Le-rho_Corr-Ineq} we  obtain
\begin{eqnarray*}
  \alpha_{\xi}(n) &=& \sup_{m\in\mathbb{R}} 
   \alpha\big(\sigma(\xi^{(k)};k\leq m),\sigma(\xi^{(k)};k\geq m+n)\big) \\
    &=& \sup_{m\in\mathbb{R}} \alpha\big(\sigma(\xi^{(m)} ),\sigma(\xi^{(m+n)})\big)\\
    &\leq& \sup_{m\in\mathbb{N}}\sup_{a,b,c,d \in L^2[0,1]} 
     \textrm{Corr}(\langle a,B^{(m)}\rangle + \langle b,X^{(m)}\rangle, \langle c,B^{(m+n)}\rangle + 
\langle d,X^{(m+n)}\rangle)\\ 
&\leq& e^{-\alpha(n-1)}\sup_{a,b,c,d,\in L^2[0,1]} 
     \textrm{Corr}(\langle a,B^{(1)}\rangle + \langle b,X^{(1)}\rangle, 
        \langle c,B^{(2)}\rangle + \langle d,X^{(2)}\rangle)\\
    &=&\mathcal{O}(e^{-\alpha(n-1)})
\end{eqnarray*}
where the last equality is stated in Lemma~\ref{le:sup-corr-order}.
Note that each $ \Delta R_k $ may be represented as 
\begin{equation*}
    \Delta R_k={\frak f}_k(\xi^{(k)})
\end{equation*}
where ${\frak f}_k:C[0,1]\times C[0,1]\rightarrow\mathbb{R}$ is a measurable function. 
Since the $\sigma$-algebra generated by $ {\frak f}_k(\xi^{(k)}):\Omega\rightarrow\mathbb{R} $ 
is smaller or equal the  
$\sigma$-algebra generated by $ \xi^{(k)}:\Omega\rightarrow C[0,1]\times C[0,1] $, 
the bound for $\alpha_{\xi}(n)$ established above is also valid for $ \alpha_{\Delta R}(n) $. 
\end{proof}

\begin{lemma} \label{Le-rho_Corr-Ineq}
Let $ (H,\langle.,.\rangle) $ be separable Hilbert-space and $ (X,Y) $ be a pair of $ H $-valued random variables
with Gaussian joint law. Then one has
$$  
\rho(\sigma(X),\sigma(Y))\leq\max_{a,b\in H}{\rm Corr}(\langle a,X\rangle,\langle b,Y\rangle) 
.$$
\end{lemma}

\begin{proof}
Let $ (e_i)_{i\in\mathbb{N}} $ be a system of  orthonormal basis vectors for the Hilbert space $ H $. 
If we set $ V_i:=\langle X,e_i\rangle $ and $ W_j:=\langle Y,e_j\rangle $ then we have the 
representations
$$ X=\sum_{i=1}^\infty V_i e_i \ \ \ \mbox{and} \ \ \ Y=\sum_{j=1}^\infty W_j e_j .$$
Note that $ \sigma(X)=\sigma(V_1,V_2,...) $ and $ \sigma(Y)=\sigma(W_1,W_2,...) $. 
It follows from Prop. 3.18 and Thm. 9.2 in \cite{Bradley} that
\begin{eqnarray*}
     \rho(\sigma(X),\sigma(Y)) 
&=& \lim_{n\rightarrow\infty}\rho(\sigma(V_1,..,V_n),\sigma(W_1,...,W_n))\\
&=& \lim_{n\rightarrow\infty}\sup_{a_1,...,a_n, b_1,...,b_n \in \mathbb{R}}
           {\rm Corr}\Big(\sum_{i=1}^na_iV_i,\sum_{j=1}^nb_jW_j\Big).
\end{eqnarray*}
Since the correlation is homogeneous we can assume without loss of generality that 
$ \sum a_i^2=1 $ and $ \sum b_j^2=1 $ holds. From this then follows 
\begin{eqnarray*}
     \rho(\sigma(X),\sigma(Y)) 
&\leq& \sup_{(a_i)_{i\in\mathbb{N}},(b_j)_{j\in\mathbb{N}}:\sum a_i^2=\sum b_j^2=1}
           {\rm Corr}\Big(\sum_{i=1}^\infty a_iV_i,\sum_{j=1}^\infty b_jW_j\Big)\\
&\leq& \sup_{a,b\in H:\|a\|=\|b\|=1}
           {\rm Corr}\Big( \langle a,X\rangle,\langle b,Y\rangle \Big).
\end{eqnarray*}
The second inequality follows since one has for $ a\in H $ and $ a_i:=\langle a,e_i\rangle $ that
$$ \langle a,X\rangle=\sum_{i=1}^\infty a_iV_i  .  $$
This finishes the proof of the lemma.
\end{proof}

\begin{lemma}\label{le:sup-corr-order}
For all $ a,b,c,d\in L^2[0,1] $ we have 
\begin{eqnarray*}
 && {\rm Corr}(\langle a,B^{(m)}\rangle + \langle b,X^{(m)}\rangle, 
         \langle c,B^{(m+n)}\rangle + \langle d,X^{(m+n)}\rangle) \\
 &=& e^{-\alpha(n-1)}{\rm Corr}(\langle a,B^{(m)}\rangle + \langle b,X^{(m)}\rangle, 
         \langle c,B^{(m+1)}\rangle + \langle d,X^{(m+1)}\rangle) .
\end{eqnarray*}
\end{lemma}

\begin{proof}
Since $ B^{(m+n)} $ is independent from $ \sigma(B^{(m)},X^{(m)}) $ we have that
\begin{eqnarray*}
 && {\rm Corr}(\langle a,B^{(m)}\rangle + \langle b,X^{(m)}\rangle, 
         \langle c,B^{(m+n)}\rangle + \langle d,X^{(m+n)}\rangle) \\
&=& {\rm Corr}(\langle a,B^{(m)}\rangle + \langle b,X^{(m)}\rangle, 
          \langle d,X^{(m+n)}\rangle) \\
 &=& \frac{\textrm{Cov}(\langle a,B^{(m)}\rangle + \langle b,X^{(m)}\rangle, 
         \langle d,X^{(m+n)}\rangle)}
   {\sqrt{\textrm{Var}(\langle a,B^{(n)}\rangle + \langle b,X^{(n)}\rangle)}
    \sqrt{\textrm{Var}(\langle d,X^{(m+n)}\rangle)}} \\
   &=& \frac{\textrm{Cov}(\langle a,B^{(m)}\rangle + \langle b,X^{(m)}\rangle, 
          \langle d,X^{(m+n)}\rangle)}
   {\sqrt{\textrm{Var}(\langle a,B^{(n)}\rangle + \langle b,X^{(n)}\rangle)}
    \sqrt{\textrm{Var}(\langle d,X^{(n)}\rangle)}} .
\end{eqnarray*}
Note that we used the fact that the sequence $ (X^{(m)})_{m\in\mathbb{N}} $ is stationary.

We will use the fact that $ (X_t)_{t\geq m} $ is the unique solution 
of the SDE (\ref{eq:gen-ou}) with initial condition $ X_{m} $ to see that 
$ X^{(m+n)} $ has the representation
\[  
X_{m+n+s}=e^{-\alpha (n+s)}X_m+h(n+s)
            +\sigma e^{-\alpha (n+s)}\int_0^{n+s}e^{\alpha r}dB_{m+r} .
\]
This representation follows from Lemma~4.2 in Dehling, Franke, Kott (2010).
We use this fact to compute the covariance in the above formula:

\begin{eqnarray*}
  && {\rm Cov}(\langle a,B^{(m)}\rangle + \langle b,X^{(m)}\rangle, \langle d,X^{(m+n)}\rangle) \\
   &=& e^{-\alpha n}{\rm Cov}\Bigg(\langle a,B^{(m)}\rangle + \langle b,X^{(m)}\rangle, 
          \int_0^1 d(s)e^{-\alpha s}X_m ds\Bigg) \\
  && +   {\rm Cov}\Bigg(\langle a,B^{(m)}\rangle + \langle b,X^{(m)}\rangle, 
          \int_0^1 d(s)h(n+s) ds\Bigg) \\
    && + e^{-\alpha n}{\rm Cov}\Bigg(\langle a,B^{(m)}\rangle + \langle b,X^{(m)}\rangle, 
          \int_0^1 d(s) \sigma e^{-\alpha s}\int_0^1e^{\alpha r}dB_{m+r} ds\Bigg) \\
  && + e^{-\alpha n}{\rm Cov}\Bigg(\langle a,B^{(m)}\rangle + \langle b,X^{(m)}\rangle, 
          \int_0^1 d(s) \sigma e^{-\alpha s}\int_1^{n+s}e^{\alpha r}dB_{m+r} ds\Bigg) \\
\end{eqnarray*}
Note that the second term on the right vanishes, since the right entry in the covariance 
is deterministic. Further, the fourth term vanishes, since the Brownian increments on 
the interval $ [m+1,m+n+s] $ are independent with respect to $ \sigma(B^{(m)},X^{(m)}) $.
We thus have 
\begin{eqnarray*}
  && {\rm Cov}(\langle a,B^{(m)}\rangle + \langle b,X^{(m)}\rangle, \langle d,X^{(m+n)}\rangle) \\
   &=& e^{-\alpha n}{\rm Cov}\Bigg(\langle a,B^{(m)}\rangle + \langle b,X^{(m)}\rangle, 
          \int_0^1 d(s)e^{-\alpha s}X_m ds\Bigg) \\
    && + e^{-\alpha n}{\rm Cov}\Bigg(\langle a,B^{(m)}\rangle + \langle b,X^{(m)}\rangle, 
          \int_0^1 d(s) \sigma e^{-\alpha s}\int_0^1e^{\alpha r}dB_{m+r} ds\Bigg) .
\end{eqnarray*}
The result follows since we can do the same reasoning for $ n=1 $. 
\end{proof}

\begin{corollary} \label{Cor-starke-Approximation}
There exists an iid-sequence of $ \mathbb{R}^{p+1} $-valued Gaussian random variables 
$ \zeta_i;i\in\mathbb{N} $ such that for $ U_k:=\sum_{i=1}^k\zeta_i $ one has almost surely
$$ R_k-U_k=O(k^{1/2-\lambda}) \ \ \ \mbox{for some $ \lambda>0 $ as $ k\rightarrow\infty $}.$$ 
\end{corollary}
\begin{proof}
This follows from Proposition \ref{Prop-R_k-mixing} and the theorem from Kuelbs and Philipp on 
strong approximation of mixing random sequences (see \cite{Philipp}).
\end{proof}

\begin{remark}
The previous corollary has its analogue in the autoregressive situation which was treated by Davis, Huang and Yao
(see formula 2.2 in \cite{davis}).
\end{remark}

\begin{remark}
In the following we will denote by $ \Gamma_{p+1} $ the covariance matrix of the Gaussian 
random variable $ \zeta_1 $. It then follows that the sequence of random variables 
$ \Gamma^{-1}_{p+1}R_{[nt]}/\sqrt{n};t\in[0,1] $ converges in distribution toward a 
$ p+1 $-dimensional Brownian motion with covariance matric $ I_{p+1} $. Here 
$ I_{p+1} $ denotes the identity matrix with $ p+1 $ rows. It follows from Proposition
\ref{prop:rtau} that $ \Gamma_{p+1}=\Sigma $. 
\end{remark}

\begin{proposition} \label{Prop-Lambda-Approx}
For all $\delta>0$ one has as $u\rightarrow 0$:
\begin{equation*}
    \limsup_{T\rightarrow \infty} \P\left(\Big|\sup_{0< s \leq u} \Lambda_T(s) 
- \sup_{0< s \leq u}   R_{sT}^t Q_{sT}^{-1}R_{sT}\Big| > a_T\delta\right) 
\rightarrow 0
\end{equation*}
and
\begin{equation*}
    \limsup_{T\rightarrow \infty} \P \left(\Big|\sup_{1-u< s \leq  1} \Lambda_T(s) 
- \sup_{1-u< s \leq 1}   (R_{T} -R_{sT})^t (Q_{T}-Q_{sT})^{-1}
(R_{T} -R_{sT}) \Big| > a_T\delta\right)\rightarrow0.
\end{equation*}

\end{proposition}

\begin{proof}
It holds that
\begin{eqnarray}
\nonumber && a_T^{-1}\big|\sup_{0< s \leq u} \Lambda_T(s) - \sup_{0< s \leq u}   
R_{sT}^t Q_{sT}^{-1}R_{sT}\big| \\
\nonumber &\leq& \sup_{0< s \leq u}a_T^{-1}\big| \Lambda_T(s) 
- R_{sT}^t Q_{sT}^{-1}R_{sT}  \big|\\
\nonumber   &=& \sup_{0< s \leq u}a_T^{-1}\big|(R_{T} -R_{sT})^t 
(Q_{T}-Q_{sT})^{-1}(R_{T} -R_{sT})   - R_{T}^t Q_{T}^{-1}R_{T}\big| \\
\label{eq:conv-sup-bound-1}  &\xrightarrow{\mathcal{D}} & \sup_{0< s \leq u} 
\bigg| \frac{\|W(1) -W(s) \|^2}{1-s} - \|W(1) \|^2  \bigg| \quad(\textrm{as }
 T\rightarrow\infty)\\
   \nonumber &\rightarrow & 0, \quad\textrm{almost surely},
\end{eqnarray}
as $u\rightarrow 0$,  where the convergence in (\ref{eq:conv-sup-bound-1})
is implicated by the proof of Theorem~\ref{theo-1} and the fact that $a_T\rightarrow 1$. 
Analogously, one has
\begin{eqnarray*}
  &&\hspace{-20mm}  a_T^{-1} \left|\sup_{1-u< s \leq  1} \Lambda_T(s) 
- \sup_{1-u< s \leq 1}   (R_{T} -R_{sT})^t (Q_{T}-Q_{sT})^{-1}
 (R_{T} -R_{sT}) \right|\\[1mm]
  &\leq& \sup_{1-u< s \leq  1}a_T^{-1} \left|R_{sT}^t Q_{sT}^{-1}R_{sT}   
- R_{T}^t Q_{T}^{-1}R_{T}   \right|\\
  &\xrightarrow{\mathcal{D}} & \sup_{1-u< s \leq  1} \left| \frac{\|W(s) \|^2}{s} - \|W(1) \|^2  \right| 
   \quad(\textrm{as }T\rightarrow\infty)\\
   \nonumber &\rightarrow & 0, \quad\textrm{almost surely},
\end{eqnarray*}
as $u\rightarrow 0$.
\end{proof}

\begin{proposition} \label{Prop-Asymp-Gumbel}
Under the framework of Theorem~\ref{theo-1}
it holds under the null hypothesis that
\begin{equation*}
    \frac{1}{a_T}\left(\sup_{0< s \leq u}   R_{sT}^t Q_{sT}^{-1}R_{sT} 
- b_T \right)\stackrel{\mathcal{D}}{\longrightarrow} G^*
\end{equation*}
and
\begin{equation*}
    \frac{1}{a_T}\left(\sup_{1-u< s \leq 1}   (R_{T} -R_{sT})^t (Q_{T}-Q_{sT})^{-1}
(R_{T} -R_{sT})\right)\stackrel{\mathcal{D}}{\longrightarrow} G^*,
\end{equation*}
as $n\rightarrow \infty$, where $G^*$ denotes a real-valued random variable satisfying
\begin{equation*}
    \P(G^*\leq x) = \exp(-e^{-x/2})
\end{equation*}
and where $a_T$ and $b_T$ are given in Proposition~\ref{le:cor-davis}. 
\end{proposition}

\begin{proof}
The reasoning follows the lines of the proof of remark A3 presented in \cite{davis} 
(see page 297). We first note that 
\begin{eqnarray*}
    R_{sT}^t\Gamma_{p+1}^{-1}R_{sT}-U_{[sT]}^t\Gamma_{p+1}^{-1}U_{[sT]} = 
    R_{sT}^t\Gamma_{p+1}^{-1}(R_{sT}-U_{[sT]})+(R_{sT}^t-U_{[sT]}^t)\Gamma_{p+1}^{-1}U_{[sT]} .
\end{eqnarray*}
The law of iterated logarithm implies $ U_{[sT]}^t\Gamma_{p+1}^{-1}=O(([sT]\log [sT])^{1/2})  $ and 
Corollary \ref{Cor-starke-Approximation} then implies 
$ R_{sT}^t\Gamma_{p+1}^{-1}=O(([sT]\log [sT])^{1/2}) $. Using those facts and Corollary 
\ref{Cor-starke-Approximation} again yields 
\begin{eqnarray} \label{Gleich-R_k-U_k-Vergleich}
R_{sT}^t\Gamma_{p+1}^{-1}R_{sT}-U_{[sT]}^t\Gamma_{p+1}^{-1}U_{[sT]} = O([sT]^{1-\lambda'}) 
\end{eqnarray}
for some $ \lambda'>0 $ as $ T\rightarrow\infty $. 
\\
Since by Proposition \ref{prop:qtau} one has 
$ Q_{sT}/sT\rightarrow\Gamma_{p+1} $ it follows that
\begin{eqnarray*}
     R_{sT}^tQ_{sT}^{-1}R_{sT}-\frac{1}{sT}R_{sT}\Gamma_{p+1}^{-1}R_{sT} 
&=&   \frac{R_{sT}^t}{(sT)^{1/2}}sT Q_{sT}^{-1}\Big(\Gamma_{p+1}
   -\frac{Q_{sT}}{sT}\Big)\Gamma_{p+1}^{-1}\frac{R_{sT}}{(sT)^{1/2}}
\longrightarrow0.
\end{eqnarray*}
This relation together with Equation (\ref{Gleich-R_k-U_k-Vergleich}) implies that as 
$ T\rightarrow\infty $ one has
\begin{eqnarray} \label{Gleich-Q_n-Gamma-Vergleich}
  R_{sT}^tQ_{sT}^{-1}R_{sT}-U_{[sT]}^t\Gamma_{p+1}^{-1}U_{[sT]}/[sT]\longrightarrow0.
\end{eqnarray}
Proposition \ref{prop:qtau} and the continuous mapping theorem yield
$$ \sup_{s\in [u,1]}R_{sT}^tQ_{sT}^{-1}R_{sT}\stackrel{\cal D}{\longrightarrow}
    \sup_{s\in[u,1]}\frac{\|W(s)\|^2}{s}  .$$
It thus follows that
$$  \sup_{s\in[u,1]}R_{sT}^tQ_{sT}^{-1}R_{sT}=O_P(1)  $$
Moreover we have
$$   \sup_{s\in(0,u]}R_{sT}^tQ_{sT}^{-1}R_{sT} \stackrel{P}{\longrightarrow}
\infty .$$
Thus with probability closer and closer to one the supremum is achieved in the interval
$ (0,u] $ and not in $ [u,1] $. It then follows that
\begin{eqnarray} \label{Gleich-Sup_ganz-Sup_teil}
  \P\Big(\sup_{s\in(0,1]}R_{sT}^tQ_{sT}^{-1}R_{sT}=
    \sup_{s\in(0,u]}R_{sT}^tQ_{sT}^{-1}R_{sT}\Big)\longrightarrow 1.
\end{eqnarray}
We also note that for a fixed $ M>0 $ one has
\begin{eqnarray} \label{Gleich-Abschneid_Unten}
  \P\Big(\sup_{s\in(0,1]}R_{sT}^tQ_{sT}^{-1}R_{sT}=
    \sup_{s\in(M/T,1]}R_{sT}^tQ_{sT}^{-1}R_{sT}\Big)\longrightarrow 1.
\end{eqnarray}
and
\begin{eqnarray} \label{Gleich-Abschneid_Unten}
  \P\Big(\sup_{s\in(0,1]}U^t_{[sT]}\Gamma_{p+1}^{-1}U_{[sT]}=
    \sup_{s\in(M/T,1]}U^t_{[sT]}\Gamma_{p+1}^{-1}U_{[sT]}\Big)\longrightarrow 1.
\end{eqnarray}
Let 
$$ R_T(M):=\sup_{s\in(M/T,1]}R^t_{sT}Q_{sT}^{-1}R^t_{sT}
                 -\sup_{s\in(M/T,1]}U^t_{[sT]}\Gamma_{p+1}^{-1}U^t_{[sT]}\big/[sT] .$$
From Equation (\ref{Gleich-Q_n-Gamma-Vergleich}) we have
\begin{eqnarray}
  |R_T(M)|&=&   \Big|\sup_{s\in(M/n,1]}R^t_{sT}Q_{sT}^{-1}R^t_{sT}
                 -\sup_{s\in(M/n,1]}U^t_{[sT]}\Gamma_{p+1}^{-1}U^t_{[sT]}\big/[sT]\Big| \\
  \nonumber  &\leq& \sup_{s\in(M/n,1]}
  \Big|R^t_{sT}Q_{sT}^{-1}R^t_{sT}-U^t_{[sT]}\Gamma_{p+1}^{-1}U^t_{[sT]}\big/[sT]\Big|
\end{eqnarray}
which goes to zero as $ M\rightarrow\infty $ uniformly in $ T\geq\nu $.
It now follows from Eq.(\ref{Gleich-Sup_ganz-Sup_teil}) and Proposition \ref{le:cor-davis} that
\begin{eqnarray*}
 &&  \lim_{T\rightarrow\infty}
          \P\Big(\sup_{s\in(0,u]}R^t_{sT}Q^{-1}_{sT}R_{sT}\leq a_Tx+b_T\Big)\\
  &=& \lim_{T\rightarrow\infty}
          \P\Big(\sup_{s\in(0,1]}R^t_{sT}Q^{-1}_{sT}R_{sT}\leq a_Tx+b_T\Big)\\
  &=& \lim_{M\rightarrow\infty}\lim_{T\rightarrow\infty}
          \P\Big(\sup_{s\in(M/T,1]}R^t_{sT}Q^{-1}_{sT}R_{sT}\leq a_Tx+b_T\Big)\\
 &=& \lim_{M\rightarrow\infty}\lim_{T\rightarrow\infty}
        \P\Big(\sup_{s\in(M/T,1]}U^t_{[sT]}\Gamma_{p+1}^{-1}U^t_{[sT]}\big/[sT]  
                 \leq a_Tx+b_T+R_T(M)\Big)\\ 
&=& \lim_{T\rightarrow\infty}
  \P\Big(\sup_{s\in(0,1]}U^t_{[sT]}\Gamma_{p+1}^{-1}U^t_{[sT]}\big/[sT]
           \leq a_Tx+b_T\Big)\\
&=& \lim_{T\rightarrow\infty}
           \P\Big(\sup_{s\in(0,1]}\|S_{[sT]}\|^2/[sT]
                \leq a_Tx+b_T\Big) \longrightarrow \exp\big(-e^{-x/2}\big).
\end{eqnarray*}
This proves the first statement of the proposition. The second one is proved in an analogous way.
\end{proof}

\begin{proof}[Proof of Theorem 2]

Since for fixed $ x\in\mathbb{R} $ one has $ a_Tx+b_T\rightarrow\infty $ as $ T\rightarrow\infty $
it follows from Theorem \ref{theo-1} for all $ u\in(0,1/2) $ that
$$ \P\Big(\sup_{u<s<1-u}\Lambda_T(s)\leq a_Tx+b_T\Big)\longrightarrow 1 .$$
Therefore, one has as $ T\rightarrow\infty $ that
\begin{eqnarray*}
  \P\Big(\sup_{0\leq s\leq 1}\Lambda_T(s)\leq a_Tx+b_T\Big)
 =  \P\Big(\sup_{0\leq s\leq u}\Lambda_T(s)\leq a_Tx+b_T,
     \sup_{1-u\leq s\leq 1}\Lambda_T(s)\leq a_Tx+b_T \Big).
\end{eqnarray*}
By Proposition \ref{Prop-Lambda-Approx} this has for $ T\rightarrow\infty $ the same limit as
$$
  \P\Bigg(\!\sup_{0<s<u}R_{sT}^tQ_{sT}^{-1}R_{sT}\leq a_Tx+b_T, \!
   \sup_{1-u<s<1}(R_{T}-R_{sT})^t(Q_{T}-Q_{sT})^{-1}(R_{T}-R_{sT}) 
  \leq a_Tx+b_T\!\Bigg) .
$$
Proposition \ref{Prop-Asymp-Gumbel} yields that this last expression converges toward 
$ \exp(-2e^{-x/2}) $ since the two sequences
$$ \sup_{0<s<u}R_{sT}^tQ_{sT}^{-1}R_{sT} $$
and  
$$ \sup_{1-u<s<1}(R_{T}-R_{sT})^t(Q_{T}-Q_{sT})^{-1}(R_{T}-R_{sT}) $$
are asymptotically independent by Proposition \ref{Prop-R_k-mixing}.
\end{proof}

{\bf Acknowledgments}  This work was partly supported by the Collaborative Research Pro\-ject SFB 823 (Statistical modelling of nonlinear dynamic processes) of the German Research Foundation DFG. Thomas Kott was supported by the E.ON Ruhrgas AG. The authors wish to thank Martin Wendler for his help with the proof of Propostion 4.2, and two anonymous referees for their 
careful reading of the manuscript and for their comments that helped to improve the paper.

\end{document}